\documentclass[a4paper,11pt,reqno]{amsart}

\usepackage{
amsfonts,
amsmath,
amsopn,
amssymb,
amsthm,
bbm,
bbold,
dsfont,
enumitem,
graphicx,
mathrsfs,
mathtools,
soul,
subfig,
verbatim,
xcolor,
xspace
}

\usepackage{geometry}
\usepackage[hidelinks]{hyperref}
\usepackage[utf8]{inputenc}

\mathtoolsset{showonlyrefs}

\newtheorem{theorem}{Theorem}[section]
\newtheorem{lemma}[theorem]{Lemma}
\newtheorem{proposition}[theorem]{Proposition}
\newtheorem{corollary}[theorem]{Corollary}

\theoremstyle{remark}
\newtheorem{remark}[theorem]{Remark}
\newtheorem*{remark*}{Remark}

\theoremstyle{definition}

\DeclarePairedDelimiter\abs{\lvert}{\rvert}%

\newcommand{\R}{\mathbb{R}}

\newcommand{\Rp}{\mathbb{R}^+}

\newcommand{\Eps}{\mathcal{E}}

\newcommand{\A}{\mathcal{A}}

\newcommand{\F}{\mathcal{F}}

\renewcommand{\leq}{\leqslant}
\renewcommand{\geq}{\geqslant}

\newcommand{\la}{\lambda}

\newcommand{\de}{\partial}

\newcommand{\sech}{\text{sech}}

\newcommand{\f}[2]{\frac{#1}{#2}}

\newcommand{\deb}{\rightharpoonup}

\title[NLS ground states on the half--line with point interactions]{NLS ground states on the half-line with point interactions} 

\author[F. Boni]{Filippo Boni$^{*}$}
\address{Università degli Studi di Napoli Federico II, Dipartimento di Matematica e Applicazioni ``Renato Caccioppoli”, Via Cintia, Monte S. Angelo, 80126, Napoli, Italy}
\thanks{$^{*}$ Corresponding author}

\email{filippo.boni@unina.it}

\author[R. Carlone]{Raffaele Carlone}
\address{Università degli Studi di Napoli Federico II,Dipartimento di Matematica e Applicazioni ``Renato Caccioppoli”, Via Cintia, Monte S. Angelo, 80126, Napoli, Italy}
\email{raffaele.carlone@unina.it}

\date{\today}

\begin{document}

%%%%%%%%%%%%%%%%%%%%%%%%%%%%%%%%%%
%%%%%%%%%% Intestazione %%%%%%%%%%
%%%%%%%%%%%%%%%%%%%%%%%%%%%%%%%%%%

\begin{abstract}
We investigate the existence and the uniqueness of NLS ground states of fixed mass on the half--line in the presence of a point interaction at the origin. The nonlinearity is of power type, and the regime is either $L^{2}$-subcritical or $L^{2}$-critical, while the point interaction is either attractive or repulsive. 
In the $L^{2}$-subcritical case, we prove that ground states exist for every  mass  value if the interaction is attractive, while ground states exist only for sufficiently large masses if the interaction is repulsive. In the latter case, if the power is less or equal to four, ground states coincide with the only bound state. If instead, the power is greater than four, then there are values of the mass for which two bound states exist, and neither of the two is a ground state, and values of the mass for which two bound states exist, and one of them is a ground state.
In the $L^{2}$-critical case, we prove that ground states exist for masses strictly below a critical mass value  in the attractive case, while ground states never exist in the repulsive case.
 
\end{abstract}

\maketitle

\vspace{-.5cm}

\noindent {\footnotesize \textul{Keywords:} standing waves, nonlinear Schr\"odinger, ground states, delta interaction.}

%%%%%%%%%%%%%%%%%%%%%%%%%%%%%%%%%%%%%%%%%%%%%%%%%%%%%%%%%%%%%%%%%%%%%%%%%%%%%%%%
%%%%%%%%%%%%%%%%%%%%%%%%%%%%%%%%%%%%%%%%%%%%%%%%%%%%%%%%%%%%%%%%%%%%%%%%%%%%%%%%
%%%%%%%%%%%%%%%%%%%%%%%%%%%%%%%%%%%%%%%%%%%%%%%%%%%%%%%%%%%%%%%%%%%%%%%%%%%%%%%%
%%%%%%%%%%%%%%%%%%%%%%%%%%%%%%%%%%%%%%%%%%%%%%%%%%%%%%%%%%%%%%%%%%%%%%%%%%%%%%%%
%%%%%%%%%%%%%%%%%%%%%%%%%%%%%%%%%%%%%%%%%%%%%%%%%%%%%%%%%%%%%%%%%%%%%%%%%%%%%%%%
%%%%%%%%%%%%%%%%%%%%%%%%%%%%%%%%%%%%%%%%%%%%%%%%%%%%%%%%%%%%%%%%%%%%%%%%%%%%%%%%
%%%%%%%%%%%%%%%%%%%%%%%%%%%%%%%%%%%%%%%%%%%%%%%%%%%%%%%%%%%%%%%%%%%%%%%%%%%%%%%%
%%%%%%%%%%%%%%%%%%%%%%%%%%%%%%%%%%%%%%%%%%%%%%%%%%%%%%%%%%%%%%%%%%%%%%%%%%%%%%%%
%%%%%%%%%%%%%%%%%%%%%%%%%%%%%%%%%%%%%%%%%%%%%%%%%%%%%%%%%%%%%%%%%%%%%%%%%%%%%%%%
%%%%%%%%%%%%%%%%%%%%%%%%%%%%%%%%%%%%%%%%%%%%%%%%%%%%%%%%%%%%%%%%%%%%%%%%%%%%%%%%
%%%%%%%%%%%%%%%%%%%%%%%%%%%%%%%%%%%%%%%%%%%%%%%%%%%%%%%%%%%%%%%%%%%%%%%%%%%%%%%%
%%%%%%%%%%%%%%%%%%%%%%%%%%%%%%%%%%%%%%%%%%%%%%%%%%%%%%%%%%%%%%%%%%%%%%%%%%%%%%%%

%%%%%%%%%%%%%%%%%%%%%%%%%%%%%%%%%%
%%%%%%%%%% Introduzione %%%%%%%%%%
%%%%%%%%%%%%%%%%%%%%%%%%%%%%%%%%%%

\section{Introduction}

In this paper, we investigate the existence and the uniqueness of ground states of the energy
\begin{equation}
\label{FR+}
F(u)=\f{1}{2}\int_0^{+\infty}|u'|^{2}\,dx-\f{1}{p}\int_0^{+\infty}|u|^{p}\,dx+\f{\alpha}{2}|u(0)|^{2}
\end{equation}
 among functions belonging to
\begin{equation}
\label{mass-const}
H^{1}_{\mu}(\R^{+}):=\left\{v\in H^{1}(\R^{+})\,:\,\int_{0}^{+\infty}|u|^{2}=\mu\right\},\quad \mu>0.
\end{equation}

\noindent We denote  with
\begin{equation}
\label{GS-lev}
\F(\mu):=\inf_{u\in H^{1}_\mu(\R^{+})} F(u)
\end{equation}
the ground state energy level and, accordingly, a \emph{ground state} $u$ of \eqref{FR+} at mass $\mu$ is defined as a global minimizer of \eqref{FR+} in the space \eqref{mass-const}, namely $u\in H^1_\mu(\Rp)$ such that $F(u)=\F(\mu)$.

In the following, the power $p$ of the nonlinear term will belong to the interval $(2,6]$, including both the $L^{2}$-subcritical case $2<p<6$ and the $L^{2}$-critical case $p=6$, and $\alpha$ will be negative or positive, corresponding to an attractive or repulsive point interaction respectively. The minimization is carried out among real-valued and positive functions. This is not restrictive since $F(|u|)\leq F(u)$ and any ground state is real-valued and positive up to a multiplication by a constant phase $e^{i\theta}$. 

By standard variational arguments, it turns out that ground states of \eqref{FR+} satisfy
\begin{equation*}
-u''-|u|^{p-2}u+\omega u+\alpha\delta_{0}u=0,
\end{equation*} 
where $\delta_0$ denotes the delta distribution at the point $0$, or, written in an  equivalent form,
\begin{equation}
\label{EL-delta-hl}
\begin{cases}
-u''-|u|^{p-2}u+\omega u=0\quad\text{on}\quad \R^{+},\\
u'(0)=\alpha u(0),
\end{cases}
\end{equation}
for some $\omega>0$. In the following, we will call \emph{bound states} all the real-valued solutions of \eqref{EL-delta-hl} belonging to $H^{1}(\R^{+})$. 

Moreover, given a solution $u$ of \eqref{EL-delta-hl}, then the associated standing wave $\psi(t,x):=e^{i\omega t}u(x)$ is a solution of the Nonlinear Schr\"odinger equation
\begin{equation*}
i\partial_{t}\psi=H_{\alpha}\psi-|\psi|^{p-2}\psi,
\end{equation*}
where $H_{\alpha}:D(H_{\alpha})\subset L^{2}(\Rp)\to L^{2}(\Rp)$ is a self-adjoint extension of $-\f{d^{2}}{dx^{2}}:C^{\infty}_{c}(\Rp)\to L^{2}(\Rp)$. Let us recall that all these self-adjoint extensions are parametrized by $\alpha\in \R\cup\{\infty\}$: in particular, the case $\alpha=0$ corresponds to the Laplacian operator with Neumann condition at the origin, while the case $\alpha=\infty$ has to be intended as the Laplacian operator with Dirichlet boundary condition. The issues under investigation in this paper involve the operator $H_{\alpha}$ with $\alpha\in \R\setminus\{0\}$, corresponding to the Laplacian operator with Robin or delta condition at the origin. { Let us underline that, when $\alpha>0$, the whole spectrum of $H_{\alpha}$ is continuous and coincides with $\sigma(H_{\alpha})=[0,+\infty)$, while, when $\alpha<0$, we have that
\begin{equation*}
\sigma(H_{\alpha})=\{-\alpha^{2}\}\cup[0,+\infty),
\end{equation*}
with $-\alpha^{2}$ the sole negative eigenvalue of $H_{\alpha}$.}

{Although our interest is in the mathematical treatment, the NLS in low dimensional manifolds has applications to the physical context. 

One of the greatest successes of the nonlinear Schr\"odinger equation is the widespread application as a model in different areas of physics, ranging from the propagation of laser beams \cite{PP} to the theory of Bose--Einstein condensates \cite{DPS} and from the signal transmission in a neuronal network \cite{CM} to fluid dynamics \cite{L}. In  some of these contexts the metric graph model on which the dynamics of the NLS is set has provided an interesting modeling.

In this context the results of this paper can be interpreted as a model where the point interaction at the origin can represent inhomogeneities of various types  in the medium and affect the dynamics itself: it is widely accepted that such defects can be mathematically introduced by means of the theory of self-adjoint extensions up to the three-dimensional case (see the book \cite{AGHH-88} for a complete discussion of this topic).}

From the mathematical point of view, {the NLSE on the half-line has been the object of several papers dealing with well--posedeness, regularity and scattering issues (see \cite{ET-17, FHM-17, FIS-05,W-05}).

More specifically,} nonlinear models with point interactions have been studied first in dimension one and only more recently in dimensions two and three. In particular, on the real line, different point interactions have been considered, including delta conditions \cite{CM-95,FJ,FOO,HMZ}, delta prime conditions \cite{ANV} and more exotic conditions such as F\"ul\"op--Tsustui type conditions \cite{FT}. In the two and three dimensional context, recent papers have addressed nonlinear problems in the presence of a point interaction, focusing both on ground states and their stability \cite{ABCT-3d, ABCT,FGI-21,GMS}  and on global well--posedness and blow--up phenomena \cite{CFN,FN}. Moreover, in  recent years such problems have been considered also on metric graphs, in presence of both linear point interactions \cite{ACFN-14, ACFN-12, ACFN-16, ACFN-JDE,SBMNU,SMSN} and nonlinear ones \cite{ABD-20,BD2,BD} (see \cite{ABR} for a review of these results).

The present paper fits in this line of research, addressing the problem of fixed-mass ground states of the NLSE on the half--line $\R^{+}$, both in the attractive and the repulsive case.

%%%%%%%%%%%%%%%%%%%%%%%%%%%%%%%%%%%%%%%%%%%%%%%%%%%%%%%%%%%%%%%%%%%%%%%%%%%%%%%%
%%%%%%%%%%%%%%%%%%%%%%%%%%%%%%%%%%%%%%%%%%%%%%%%%%%%%%%%%%%%%%%%%%%%%%%%%%%%%%%%
%%%%%%%%%%%%%%%%%%%%%%%%%%%%%%%%%%%%%%%%%%%%%%%%%%%%%%%%%%%%%%%%%%%%%%%%%%%%%%%%
%%%%%%%%%%%%%%%%%%%%%%%%%%%%%%%%%%%%%%%%%%%%%%%%%%%%%%%%%%%%%%%%%%%%%%%%%%%%%%%%
%%%%%%%%%%%%%%%%%%%%%%%%%%%%%%%%%%%%%%%%%%%%%%%%%%%%%%%%%%%%%%%%%%%%%%%%%%%%%%%%
%%%%%%%%%%%%%%%%%%%%%%%%%%%%%%%%%%%%%%%%%%%%%%%%%%%%%%%%%%%%%%%%%%%%%%%%%%%%%%%%
\subsection{Main results}

%\begin{definition}
%A bound state $u$ is orbitally stable if for any $\ep>0$ there exists $\delta>0$ such that for any $\psi_{0}\in H^{1}(\R^{+})$ satisfying $\|\psi_{0}-u\|_{H^{1}(\R^{+})}<\delta$, the only solution $\psi(t)$ of \eqref{} in $[0,T^{*})$ can be continued to a solution on $[0,+\infty)$ and
%\begin{equation*}
%\sup_{t\in \R^{+}}\inf_{\eta\in\R}\|\psi(t)-e^{i\eta}u\|_{H^{1}(\R^{+})}<\ep.
%\end{equation*}
%Otherwise the bound state is called unstable.
%\end{definition}
Let us present here the main results of the paper.  Let us point out that, in view of Proposition \ref{compactth} and Corollary \ref{compactcor}, the existence of ground states at mass $\mu$ can be reduced to the problem of comparing the ground state energy level \eqref{GS-lev} with the ground state energy level of the "problem at infinity", as defined in \cite{L-ANIHPC84}. In our case, it corresponds to the standard NLS energy on the line
\begin{equation}
\label{en-sol}
\frac{1}{2}\int_\R\abs{u'(x)}^2\,dx-\frac{1}{p}\int_\R\abs{u(x)}^p\,dx,    
\end{equation}
whose ground states at mass $\mu$ are the soliton $\phi_{\mu}$, defined in \eqref{phi-mu}, and its translations.

The first theorem investigates the subcritical case $2<p<6$ in presence of an attractive point interaction, i.e. $\alpha<0$. 

\begin{theorem}[Ground states for $2<p<6$ and $\alpha<0$]
\label{sigmaneg-p<6}
Let $2<p<6$ and $\alpha<0$. Then for every $\mu>0$ there exists a unique positive ground state of \eqref{FR+} at mass $\mu$. 
\end{theorem}

{Ground states for the attractive NLSE in presence of an attractive point interaction have been studied on star graphs in \cite{ACFN-14} and on more general metric graphs and with more general potential terms in \cite{C-18,CFN-17}. In particular, the result of Theorem \ref{sigmaneg-p<6} has been proved in \cite[Theorem 1]{ACFN-14} for star graphs with at least three half--lines, but the same proof can be repeated for the half--line, intended as a star graph with only one half--line.}
As one may expect, the existence of ground states holds for every value of the mass. This is natural since the presence of an attractive delta interaction has the effect to lower the energy of half of the soliton of double mass, that is the ground state of the standard NLS energy on $\R^{+}$ (see Section \ref{subRR+})

The next two theorems deal instead with the subcritical case $2<p<6$ in presence of a repulsive point interaction, i.e. $\alpha>0$, unravelling different phenomena with respect to the same problem on $\R$.

\begin{theorem}[Ground states for $2<p\leq 4$ and $\alpha>0$]
\label{sigmapos-pleq4}
Let $2<p\leq 4$ and $\alpha>0$. Then ground states of \eqref{FR+} at mass $\mu$ exist if and only if $\mu>\|\phi_{\alpha^{2}}\|_{L^{2}(\R)}^{2}$, where $\phi_{\alpha^{2}}$ is the soliton at frequency $\alpha^{2}$. Moreover, whenever they exist, the positive ground states are unique and coincide with the only positive bound state of mass $\mu$.
\end{theorem}

\begin{theorem}[Ground states for $4<p<6$ and $\alpha>0$] 
\label{sigmapos-p>4}
Let $4<p<6$ and $\alpha>0$. Then there exists $\mu^{*}=\mu^{*}(\alpha)<\|\phi_{\alpha^{2}}\|_{L^{2}(\R)}^{2}$ such that bound states of mass $\mu$ exist if and only if $\mu\geq \mu^{*}$. In particular, two positive bound states of mass $\mu$ exist if and only if $\mu^{*}<\mu<\|\phi_{\alpha^{2}}\|_{L^{2}(\R)}^{2}$.

Moreover, there exists $\widetilde{\mu}=\widetilde{\mu}(\alpha)$ satisfying $\mu^{*}<\widetilde{\mu}<\|\phi_{\alpha^{2}}\|_{L^{2}(\R)}^{2}$ such that ground states of \eqref{FR+} at mass $\mu$ exist if and only if $\mu\geq \widetilde{\mu}$. Whenever they exist, the positive ground states are unique.
\end{theorem}

It is well known that in presence of a repulsive point interaction no ground states exist on $\R$, since every function of a given mass $\mu$ has greater energy than the ground state energy level of \eqref{en-sol}. The situation {substantially} changes passing from $\R$ to $\R^{+}$, where ground states start existing when the mass is sufficiently large, as shown in Theorem \ref{sigmapos-pleq4} and \ref{sigmapos-p>4}. 

Let us highlight that the situation is qualitatively different depending on the strength of the power $p$. If $2<p\leq 4$, then bound states and ground states start existing together and coincide for masses larger than $\|\phi_{\alpha^{2}}\|_{2}^{2}$. If instead $4<p<6$, then there are some values of the mass between $\mu^{*}$ and $\widetilde{\mu}$ such that two positive bound states exist, but neither of the two is a ground state, and other values of the mass such that two positive bound states exist and one of them is actually the ground state.

{Let us point out that when $\alpha>0$ a different behaviour between the cases $2<p\leq 4$ and $4<p<6$ has been shown also in \cite{FJ}, where the authors studied the orbital stability of the even bound states on $\R$ in presence of a repulsive delta interaction: as a consequence, by exploiting the symmetry properties of the bound states in \cite{FJ} analogous stability results hold also on $\Rp$.}

In the next proposition, we investigate how the existence of ground states depends on the strength of the interaction $\alpha$. In particular, after fixing the value of the mass, we show that ground states exist if and only if the interaction is either attractive or repulsive with strength less than a threshold: the actual threshold is different when $p\leq 4$ and $p>4$ as a consequence of Theorem \ref{sigmapos-pleq4} and Theorem \ref{sigmapos-p>4}

\begin{proposition}
\label{GSalpha}
Let $2<p<6$ and $\mu>0$. Therefore, denoted by 
\begin{equation*}
\gamma_{p}:=\left(\f{2}{p}\right)^{\f{2}{6-p}}\left(\f{p-2}{4\int_{0}^{1}(1-s^{2})^{\f{4-p}{p-2}}}\right)^{\f{p-2}{6-p}},
\end{equation*}
there results that:
\begin{itemize}
\item[$(i)$] if $p\leq 4$, then ground states of \eqref{FR+} at mass $\mu$ exist if and only if $\alpha<\gamma_{p}\mu^{\f{p-2}{6-p}}$,
\item[$(ii)$] if $p>4$, then ground states of \eqref{FR+} at mass $\mu$ exist if and only if $\alpha\leq\widetilde{h}(\mu)$, with $\widetilde{h}(\mu)>\gamma_{p}\mu^{\f{p-2}{6-p}}$.
\end{itemize}
\end{proposition}

The next two theorems deal instead with the critical case $p=6$.  {Let us specify that most of the statement of the first theorem follows directly combining the more general result \cite[Theorem 2]{C-18} and the fact that the sharp Gagliardo-Nirenberg on $\Rp$ is well known, but we prefer to report it here to make the presentation exhaustive.} 

\begin{theorem}[Ground states for $p=6$ and $\alpha<0$]
\label{sigmaneg-p6}
Let  $p=6$ and $\alpha<0$. Then
\begin{equation}
\F(\mu)=
\begin{cases}
-c\quad&\text{if}\quad 0<\mu< \f{\sqrt{3}\pi}{4}\\
-\infty &\text{if}\quad \mu\geq \f{\sqrt{3}\pi}{4},
\end{cases}
\end{equation}
with $c>0$. Furthermore, if $0<\mu< \f{\sqrt{3}\pi}{4}$, then  ground states of \eqref{FR+} at mass $\mu$ exist and coincide with the only positive bound state.
\end{theorem}

\begin{theorem}[Ground states for $p=6$ and $\alpha>0$]
\label{sigmapos-p6}
Let $p=6$ and $\alpha>0$. Then 
\begin{equation}
\F(\mu)=
\begin{cases}
0\quad&\text{if}\quad 0<\mu\leq \f{\sqrt{3}\pi}{4}\\
-\infty &\text{if}\quad \mu> \f{\sqrt{3}\pi}{4}
\end{cases}
\end{equation}
and ground states of \eqref{FR+} at mass $\mu$ do not exist for any $\mu>0$.
\end{theorem}

In both Theorem \ref{sigmaneg-p6} and Theorem \ref{sigmapos-p6}, we observe that the critical value of the mass $\mu^{*}=\f{\sqrt{3}\pi}{4}$, below which the energy is bounded from below and above which the energy becomes unbounded, is the same as for the energy \eqref{FR+} with $\alpha=0$, i.e. without point interaction. Nevertheless, some new phenomena appear. In particular, while for $\alpha=0$ ground states exist for $\mu=\mu^{*}$ only (see Section \ref{subRR+}) {and, in presence of an attractive point interaction, they exist only for masses strictly below the critical mass $\mu=\mu^{*}$, if one adds a repulsive point interaction, then ground states do not exist for any value of the mass.}

%{\color{red} La parte sulla stabilitÃ  orbitale che avevo giÃ  iniziato a scrivere l'ho commentata: la motivazione Ãš che, dopo aver parlato con Riccardo, ci siamo resi conto che il lavoro di Fukuizumu Jeanjean del 2008 tratta i ground state dell'azione sulla retta con delta repulsiva restringendosi alle funzioni pari. Mi pare che si possa passare dal loro problema al relativo problema sulla semiretta a meno di un cambiamento della forza dell'interazione, e viceversa. Quindi tutta la parte su azione e stabilitÃ  orbitale Ãš chiaramente in overlap con quel lavoro: si potrebbe pensare ad un remark. La parte sull'energia invece no.}
\subsection*{Organization of the paper}

\begin{itemize}
\item %[(a)] 
Section \ref{prel} introduces some preliminary results concerning the standard NLS in dimension one both in the $L^{2}$-subcritical and critical case; 
\item %[(b)] 
Section \ref{char-bs} collects some useful results about bound states;
 \item %[(c)] 
Section \ref{thm-subcase} contains the proofs of Theorem \ref{sigmaneg-p<6}, Theorem \ref{sigmapos-pleq4}, Theorem \ref{sigmapos-p>4} and Proposition \ref{GSalpha};
 \item %[(d)] 
Section \ref{thm-critcase} contains the proofs of Theorem \ref{sigmaneg-p6} and Theorem \ref{sigmapos-p6}.
\end{itemize}

\subsection*{Notation} 
In the following, when this does not create confusion we use the shortened notation $\|u\|_{q}$ to denote $\|u\|_{L^{q}(\R^{+})}$ for every $q\in [2,+\infty]$.

\section{Preliminaries}
\label{prel}
Given $X=\R, \R^{+}$, the minimization problem 
\begin{equation*}
\label{infE}
\Eps(\mu,X):=\inf_{v\in H^1_\mu(X)} E(v,X)\,,
\end{equation*}
where $E(\cdot,X):H^1(X)\to\R$ is the standard NLS energy functional
\begin{equation}
\label{E}
E(u,X):=\frac{1}{2}\int_X\abs{u'(x)}^2\,dx-\frac{1}{p}\int_X\abs{u(x)}^p\,dx,
\end{equation}
with $2<p\leq 6$, is nowadays classical. Let us recall in the following the basic results concerning ground states on $\R$ and $\R^{+}$ both in the subcritical case $2<p<6$ and in the critical case $p=6$.

\subsection{Subcritical NLSE on $\R$ and on $\R^{+}$}
\label{subRR+}
The peculiarity of the subcritical case $2<p<6$ is that the energy \eqref{E} is bounded from below in $H^{1}_{\mu}(X)$  for every $\mu>0$. This is a consequence of the application of the so-called Gagliardo-Nirenberg inequalities, i.e. for every $p>2$ 
\begin{equation}
\label{GNp}
\|u\|_{L^{p}(X)}^{p}\leq K_{p} (X)\|u'\|_{L^{2}(X)}^{\f{p}{2}-1}\|u\|_{L^{2}(X)}^{\f{p}{2}+1}\quad \forall \,u\in H^{1}(X),
\end{equation}
with 
\begin{equation}
\label{defKp}
K_{p}(X):=\sup_{\substack{u\in H^{1}(X), \\u\not\equiv 0}}\f{\|u\|_{L^{p}(X)}^{p}}{\|u'\|_{L^{2}(X)}^{\f{p}{2}-1}\|u\|_{L^{2}(X)}^{\f{p}{2}+1}}<+\infty.
\end{equation}
Since it will be useful in the following, we recall also the Gagliardo-Nirenberg inequality when $p=+\infty$, that is
\begin{equation}
\label{GNinf}
\|u\|_{L^{\infty}(X)}^{2}\leq K_{\infty}(X) \|u'\|_{L^{2}(X)}\|u\|_{L^{2}(X)} \quad \forall u\in H^{1}(X),
\end{equation}
with 
\begin{equation*}
K_{\infty}(X):=\sup_{\substack{u\in H^{1}(X), \\u\not\equiv 0}}\f{\|u\|_{L^{\infty}(X)}^{2}}{\|u'\|_{L^{2}(X)}\|u\|_{L^{2}(X)}}=
\begin{cases}
1\quad&\text{if}\quad X=\R,\\
2 &\text{if}\quad X=\R^{+}.
\end{cases}
\end{equation*}

Let us first consider the case $X=\R$. Standard variational arguments show that ground states are solutions to the stationary nonlinear Schr\"odinger equation
\begin{equation}
\label{nlse}
-u''-|u|^{p-2}u+\omega u=0\qquad\text{on }\R
\end{equation}
for some $\omega>0$. In fact, for every $\omega>0$ the unique (up to translations) positive solution in $H^{1}(\R)$ of \eqref{nlse} is
\begin{equation}
\label{soliton}
\phi_\omega(x)=\left[\frac{p}{2}\omega\left(\sech^2\left(\left(\frac{p}{2}-1\right)\sqrt{\omega}|x|\right)\right)\right]^{\frac{1}{p-2}}\,.
\end{equation}
{ The mass of $\phi_\omega$ is given explicitly by
\begin{equation}
\label{mass phi w}
\|\phi_\omega\|_{L^{2}(\R)}^2=\frac{4\left(\frac{p}{2}\right)^\frac{2}{p-2}}{p-2}\omega^{\frac{6-p}{2(p-2)}}\int_{0}^1(1-s^2)^{\frac{4-p}{p-2}}\,ds,
\end{equation}
which is a continuous, strictly increasing and unbounded function of $\omega$.}
%Therefore, 
Moreover, for every $2<p<6$ and $\mu>0$ there exists a unique $\omega=\omega(\mu)$ such that $\phi_{\mu}:=\phi_{\omega(\mu)}$ is the unique (up to translations) positive ground state of \eqref{E} in $H_\mu^1(\R)$. Such ground states are called solitons, and their dependence on $\mu$ is given by 
\begin{equation}
	\label{phi-mu}
	\phi_{\mu}(x)=C_p\mu^{\f{2}{6-p}} \sech^{\f{2}{p-2}}\left(c_p\mu^\beta x\right)\,,\qquad \beta:=\f{p-2}{6-p}\,,
\end{equation}
where $C_p,\,c_p>0$ depends on $p$ only, and one can easily compute
\begin{equation}
	\label{E phi mu}
	\Eps(\mu,\R)=E(\phi_{\omega(\mu)},\R)=-\theta_p\mu^{2\beta+1}
\end{equation}
where $\theta_p>0$ depends on $p$ only {and satisfies the relation 
\begin{equation}
\label{thetap}
2\theta_{p}(2\beta+1)=\omega(1),
\end{equation}
with $\omega(1)$ being the only $\omega$ corresponding to $\|\phi_\omega\|_{L^{2}(\R)}^2=1$ in \eqref{mass phi w}.}

In the case $X=\R^{+}$, ground states solve 
\begin{equation}
\label{nlseR+}
\begin{cases}
-u''-|u|^{p-2}u+\omega u=0\qquad\text{on }\R^{+}\\
u'(0^{+})=0,
\end{cases}
\end{equation}
and  the unique positive ground state of $E(\cdot,\R^{+})$  belonging to $H^{1}_{\mu}(\R^{+})$ is given by half of the soliton of mass $2\mu$ and 
\begin{equation}
	\label{E mu R+}
	\Eps(\mu,\R^{+})=\f{1}{2}E(\phi_{2\mu},\R)=-\theta_p2^{2\beta}\mu^{2\beta+1}<-\theta_p\mu^{2\beta+1}=\Eps(\mu,\R).
\end{equation}

\subsection{Critical NLSE on $\R$ and on $\R^{+}$}
In the critical case $p=6$, it is well known that
\begin{equation*}
\Eps(\mu,X)=
\begin{cases}
0\quad\text{if}\quad \mu\leq \sqrt{\f{3}{K_{6}(X)}}, \\
-\infty\quad\text{if}\quad\mu> \sqrt{\f{3}{K_{6}(X)}},
\end{cases}
\end{equation*}
and ground states exist if and only if $\mu=\sqrt{\f{3}{K_{6}(X)}}$. In particular, $K_{6}(\R^{+})=4 K_{6}(\R)=\f{16}{\pi^{2}}$ and the supremum in definition \eqref{defKp} is realized by the solitons (up to translations) or half of the solitons
\begin{equation}
\label{sol-p6}
\phi_{\omega}(x)=\left(3\omega \sech^{2}\left(2\sqrt{\omega}x\right)\right)^{\f{1}{4}} 
\end{equation} 
respectively, whose mass is equal to $\sqrt{\f{3}{K_{6}(X)}}$ for every $\omega>0$. In particular, the solitons in \eqref{sol-p6} are all the positive solutions of \eqref{nlse} or \eqref{nlseR+}, hence positive solutions exist if and only if $\mu=\sqrt{\f{3}{K_{6}(X)}}$.

\section{Properties of bound states}
\label{char-bs}
This section collects some useful results about bound states of \eqref{EL-delta-hl}, starting from the next proposition.

\begin{proposition}
\label{stat-fix-om}
Let $p>2$, $\alpha\in \R\setminus\{0\}$ and $\omega>0$.
If $0<\omega\leq\alpha^{2}$, then \eqref{EL-delta-hl} does not admit any bound state.

If $\omega>\alpha^{2}$, then $\eta^{\omega,\alpha}(\cdot)=\phi_{\omega}(\cdot-a)$ is the only bound state of \eqref{EL-delta-hl} and is positive, up to a change of sign, with 

\begin{equation}
\label{a}
a:=\f{2\tanh^{-1}(\f{\alpha}{\sqrt{\omega}})}{(p-2)\sqrt{\omega}}.
\end{equation}

\end{proposition}
\begin{proof}
Every positive solution of \eqref{EL-delta-hl} has to coincide with a proper translation of the soliton $\phi_{\omega}$ in order to satisfy the boundary condition at the origin. In particular, if one consider a solution $u(x)=\phi_{\omega}(x-a)$, then the condition $u'(0)=\alpha u(0)$ becomes
\begin{equation}
\label{u'=alfau}
\sqrt{\omega}\tanh\left(\f{p-2}{2}\sqrt{\omega}a\right)=\alpha.
\end{equation}
If $\omega\le \alpha^{2}$, then the modulus of the left hand-side is strictly less than the modulus of the right hand-side, hence there is no $a\in \R$ satisfying \eqref{u'=alfau}. If instead $\omega>\alpha^{2}$, then by the monotonicity properties of the function $\tanh(\cdot)$ there exists an only $a\in \R$ for which \eqref{u'=alfau} holds. 
\end{proof}

In view of Proposition \ref{stat-fix-om}, let us define for every $p>2$ the function
\begin{equation}
\label{defm-alpha}
\begin{aligned}
M\colon(\alpha^{2},+\infty)&\times \R\setminus\{0\}&\to (0,+\infty)\\
(\omega&,\alpha)\quad&\mapsto \|\eta^{\omega,\alpha}\|_2^{2}.
\end{aligned}
\end{equation}

Since the parameter $\alpha\in \R\setminus\{0\}$ will often be fixed, it is convenient to denote the bound states $\eta^{\omega,\alpha}$ simply by $\eta^{\omega}$ and, with a slight abuse of notation, to consider the function
\begin{equation}
\label{defmu}
\begin{aligned}
M\colon(\alpha^{2},+\infty)&\to (0,+\infty)\\
\omega\quad&\mapsto {M}(\omega,\alpha).
\end{aligned}
\end{equation}
Moreover, we denote 
\begin{equation}
\label{A-mu}
\A_{\mu}:=\{\eta^{\omega}\,:\,M(\omega)=\mu\},
\end{equation}
i.e. the set of all the positive bound states of mass $\mu$. The cardinality of $\A_{\mu}$ will be denoted by $|\A_{\mu}|$ in the following.

\subsection{Bound states in the subcritical case}

The next proposition collects some properties of the function \eqref{defmu} when $2<p<6$.
\begin{proposition}
\label{mu-study}
Let $2<p<6$ and $\alpha\in \R\setminus\{0\}$. Then the function $M$ in \eqref{defmu} is of class $C^{1}((\alpha^{2},+\infty))$ and
\begin{equation}
\label{mass-eta}
M(\omega)=\frac{p^\frac{2}{p-2}}{2^{\f{4-p}{p-2}}(p-2)}\omega^{\frac{6-p}{2(p-2)}}\int_{-\f{\alpha}{\sqrt{\omega}}}^1(1-s^2)^{\frac{4-p}{p-2}}\,ds.
\end{equation}
Moreover:
\begin{itemize}
\item[$(i)$]if $\alpha<0$, then $M'(\omega)>0$ for every $\omega\in(\alpha^{2},+\infty)$ and $M((\alpha^{2},+\infty))=\left(0,+\infty\right)$,
\item[$(ii)$] if $\alpha>0$ and $p\leq 4$, then $M'(\omega)>0$ for every $\omega\in(\alpha^{2},+\infty)$ and 

$M((\alpha^{2},+\infty))=\left(\|\phi_{\alpha^{2}}\|_{L^{2}(\R)}^{2},+\infty\right)$,
%\[
%\lim_{\omega\to \alpha^{2}}\mu(\omega)=\|\phi_{\alpha^{2}}\|_{L^{2}(\R)}^{2},\quad \lim_{\omega\to +\infty}\mu(\omega)=+\infty,\quad \mu'(\omega)>0\quad\forall\,\omega>\alpha^{2},
%\]
\item[$(iii)$] if $\alpha>0$ and $p>4$, then 
\[
\lim_{\omega\to \alpha^{2}}M(\omega)=\|\phi_{\alpha^{2}}\|_{L^{2}(\R)}^{2},\quad \lim_{\omega\to +\infty}M(\omega)=+\infty
\]
and there exists $\omega^{*}=\omega^{*}(\alpha)>\alpha^{2}$ such that $M'(\omega)<0$ if $\alpha^{2}<\omega<\omega^{*}$, $M'(\omega)=0$ if $\omega=\omega^{*}$ and $M'(\omega)>0$ if $\omega>\omega^{*}$. In particular, the frequency $\omega^{*}=\omega^{*}(\alpha)$ is the only solution of the equation
\begin{equation}
\label{omega*}
\f{6-p}{2}M(\omega,\alpha)=\left(\f{p}{2}\right)^{\f{2}{p-2}}\alpha(\omega-\alpha^{2})^{\f{4-p}{p-2}}.
\end{equation}
\end{itemize}
\end{proposition}
\begin{proof}
Expression \eqref{mass-eta} is obtained by direct computation making use of \eqref{a}. Once one has \eqref{mass-eta}, one can compute the limits of $M(\omega)$ as $\omega\to \alpha^{2}$ and $\omega\to +\infty$. Moreover,
\begin{equation}
\label{mu prime}
\begin{split}
M'(\omega)&=\f{p^{\f{2}{p-2}}}{2^{\f{2}{p-2}}(p-2)}\omega^{\f{10-3p}{2(p-2)}}\left[\f{6-p}{p-2}\int_{-\f{\alpha}{\sqrt{\omega}}}^{1}(1-s^{2})^{\f{4-p}{p-2}}\,ds-\alpha\omega^{\f{p-6}{2(p-2)}}(\omega-\alpha^{2})^{\f{4-p}{p-2}}\right]\\
&=\f{1}{(p-2)\omega}\left[\f{6-p}{2}M(\omega)-\left(\f{p}{2}\right)^{\f{2}{p-2}}\alpha(\omega-\alpha^{2})^{\f{4-p}{p-2}}\right]
\end{split}
\end{equation} 
If $\alpha<0$, then $M'(\omega)>0$ for every $\omega>\alpha^{2}$, entailing $(i)$.
If instead $\alpha>0$, the sign of $M'(\omega)$ depends on $p>2$ and is the same of
\begin{equation*}
f(\omega):=\f{6-p}{p-2}\int_{-\f{\alpha}{\sqrt{\omega}}}^{1}(1-s^{2})^{\f{4-p}{p-2}}\,ds-\alpha\omega^{\f{p-6}{2(p-2)}}(\omega-\alpha^{2})^{\f{4-p}{p-2}},
\end{equation*}
whose derivative is
\begin{equation*}
f'(\omega)=-\f{4-p}{p-2}\alpha\omega^{\f{p-6}{2(p-2)}}(\omega-\alpha^{2})^{\f{2(3-p)}{p-2}}.
\end{equation*}

On the one hand, if $2<p<4$, then $\lim_{\omega\to \alpha^{2}} f(\omega)=2\lim_{\omega\to+\infty} f(\omega)>0$ and $f'(\omega)<0$, hence $f(\omega)>0$  and 
\[
M'(\omega)=\f{p^{\f{2}{p-2}}}{2^{\f{2}{p-2}}(p-2)}\omega^{\f{10-3p}{2(p-2)}}f(\omega)>0
\]
 for every $\omega>\alpha^{2}$. If instead $p=4$, then $M'(\omega)=\f{1}{\sqrt{\omega}}$, hence $(ii)$ follows.
 
On the other hand, if $4<p<6$, then $\lim_{\omega\to \alpha^{2}} f(\omega)=-\infty$, $\lim_{\omega\to +\infty} f(\omega)>0$ and $f'(\omega)>0$, so that there exists a unique solution $\omega^{*}>\alpha^{2}$ of the equation $f(\omega)=0$. As a consequence, $M'(\omega)<0$ for $\alpha^{2}<\omega<\omega^{*}$, $M'(\omega^{*})=0$ and $M'(\omega)>0$ for $\omega>\omega^{*}$, entailing $(iii)$. 
\end{proof}

The next corollary shows how the number of positive bound states of fixed mass $\mu>0$ depends on  $\alpha$,  $p$ and $\mu$. 
\begin{corollary}
\label{station-numb}
Let $2<p<6$ and $\alpha\in \R\setminus\{0\}$. Therefore:
\begin{itemize}
\item[$(i)$]if $\alpha<0$, then $|\A_{\mu}|=1$ for every $\mu>0$;
\item[$(ii)$] if $\alpha>0$ and $2<p\le 4$, then 
\begin{equation*}
|\A_{\mu}|=
\begin{cases}
0\quad \text{if}\quad 0<\mu\leq\|\phi_{\alpha^{2}}\|_{L^{2}(\R)}^{2},\\
1\quad \text{if}\quad \mu>\|\phi_{\alpha^{2}}\|_{L^{2}(\R)}^{2};
\end{cases}
\end{equation*}
\item[$(iii)$] if $\alpha>0$ and $p>4$, then 
\begin{equation*}
|\A_{\mu}|=
\begin{cases}
0\quad \text{if}\quad 0<\mu<\mu^{*},\\
1\quad \text{if}\quad \mu=\mu^{*},\\
2\quad \text{if}\quad \mu^{*}<\mu<\|\phi_{\alpha^{2}}\|_{L^{2}(\R)}^{2},\\
1\quad \text{if}\quad \mu\geq\|\phi_{\alpha^{2}}\|_{L^{2}(\R)}^{2},
\end{cases}
\end{equation*}
with $\mu^{*}=\mu^{*}(\alpha):=M(\omega^{*}(\alpha),\alpha)$ and $\omega^{*}(\alpha)$ being the only solution of \eqref{omega*}.
\end{itemize}
\end{corollary}
\begin{proof}
The proof of $(i)$, $(ii)$ and $(iii)$ is a straightforward consequence of $(i)$, $(ii)$ and $(iii)$ of Proposition \ref{mu-study} respectively.
\end{proof}

\begin{remark}
\label{2-stat}
The only regime in which more than one positive bound state of mass $\mu$ exists is when $\alpha>0$, $4<p<6$ and $\mu^{*}<\mu<\|\phi_{\alpha^{2}}\|_{L^{2}(\R)}^{2}$. In particular, there exist $\omega_{1}, \omega_{2}>\alpha^{2}$, with $\omega_{1}<\omega^{*}<\omega_{2}$, such that $\eta^{\omega_{1}}$, $\eta^{\omega_{2}}\in \A_{\mu}$.
\end{remark}

\begin{lemma}
Let $p>2$ and $\alpha\in \R\setminus\{0\}$. Then for every $\omega>\alpha^{2}$ there results 
\begin{equation}
\label{F-omega}
\begin{split}
F(\eta^{\omega})&=-\f{6-p}{2(p+2)}\omega M(\omega,\alpha)+\f{\alpha (p-2)}{2(p+2)}\left(\f{p}{2}(\omega-\alpha^{2})\right)^{\f{2}{p-2}}\\
&=\f{\left(\f{p}{2}\right)^{\f{2}{p-2}}}{(p+2)}\left(-\f{6-p}{p-2}\omega^{\f{p+2}{2(p-2)}}\int_{-\f{\alpha}{\sqrt{\omega}}}^1(1-s^2)^{\frac{4-p}{p-2}}\,ds+\alpha\f{p-2}{2}(\omega-\alpha^{2})^{\f{2}{p-2}}\right).
\end{split}
\end{equation}
\end{lemma}
\begin{proof}
Since $\eta^{\omega}$ solves \eqref{EL-delta-hl}, multiplying the equation by $\left(\eta^{\omega}\right)'$, integrating on $[x,+\infty)$ for every $x\in [0,+\infty)$ and integrating again on $[0,+\infty)$ one gets
\begin{equation}
\label{Cons-en}
\f{1}{2}\|{(\eta^{\omega})'}\|_2^{2}+\f{1}{p}\|\eta^{\omega}\|_p^{p}-\f{\omega}{2}\|\eta^{\omega}\|_2^{2}=0.
\end{equation}
Moreover, multiplying the first line of \eqref{EL-delta-hl} by $\eta^{\omega}$ and making use of the second line of \eqref{EL-delta-hl}, there results
\begin{equation}
\label{Nehari-M}
\|(\eta^{\omega})'\|_2^{2}-\|\eta^{\omega}\|_p^{p}+\alpha|\eta^{\omega}(0)|^{2}+\omega\|\eta^{\omega}\|_2^{2}=0.
\end{equation}
By using \eqref{Cons-en} and \eqref{Nehari-M}, we get 
\begin{equation*}
F(\eta^{\omega})=-\f{6-p}{2(p+2)}\omega\|\eta^{\omega}\|_2^{2} +\f{\alpha (p-2)}{2(p+2)}\left|\eta^{\omega}(0)\right|^{2},
\end{equation*}
hence \eqref{F-omega} follows.
\end{proof}

The next proposition deals with the case $2<p<6$ and $\alpha>0$ and establishes which positive bound state between $\eta^{\omega_{1}}$ and $\eta^{\omega_{2}}$ has least energy.
\begin{proposition}
\label{station-least-en}
Let $4<p<6$ and $\alpha>0$. Then $F(\eta^{\omega})$ is a strictly decreasing function of $\omega>\alpha^{2}$.

Moreover, given $\mu^{*}<\mu<\|\phi_{\alpha^{2}}\|_{L^{2}(\R)}^{2}$, then the two positive bound states $\eta^{\omega_{1}}$ and $\eta^{\omega_{2}}$ of mass $\mu$, with $\alpha^{2}<\omega_{1}<\omega^{*}<\omega_{2}$, satisfy
\[
F\left(\eta^{\omega_{1}}\right)>F\left(\eta^{\omega_{2}}\right).
\]
\end{proposition}
\begin{proof}
Fix $\mu^{*}<\mu<\|\phi_{\alpha^{2}}\|_{L^{2}(\R)}^{2}$ and consider a positive bound state $\eta^{\omega}$ of mass $\mu$. Then, computing the derivative of \eqref{F-omega} with respect to $\omega$, one gets
\begin{equation}
\label{dF domega}
\f{d}{d\omega}F(\eta^{\omega})=\f{p^{\f{2}{p-2}}}{2^{\f{p}{p-2}}(p-2)}\left[-\f{6-p}{p-2}\omega^{\f{6-p}{2(p-2)}}\int_{-\f{\alpha}{\sqrt{\omega}}}^1(1-s^2)^{\frac{4-p}{p-2}}\,ds+\alpha(\omega-\alpha^{2})^{\f{4-p}{p-2}}\right].
\end{equation}
Denoting by $g(\omega):=-\f{6-p}{p-2}\omega^{\f{6-p}{2(p-2)}}+\alpha(\omega-\alpha^{2})^{\f{4-p}{p-2}}$, we observe that 
\[
\f{d}{d\omega}F(\eta^{\omega})\le  \f{p^{\f{2}{p-2}}}{2^{\f{p}{p-2}}(p-2)}g(\omega).
\]
In this regard, $\lim_{\omega\to(\alpha^{2})^{+}} g(\omega)<0$ and, using the fact that $p>4$,
\begin{equation*}
g'(\omega)=-\f{(6-p)^{2}}{2(p-2)^{2}}\omega^{\f{6-p}{2(p+2)}-1}-\alpha\f{p-4}{p-2}(\omega-\alpha^{2})^{\f{4-p}{p-2}-1}<0,
\end{equation*}
so that $g(\omega)<0$ for every $\omega>\alpha^{2}$ and $F(\eta^{\omega})$ is a strictly decreasing function of $\omega>\alpha^{2}$, entailing the thesis.
\end{proof}

\begin{remark}
\label{unique-stat}
In view of Corollary \ref{station-numb} and Proposition \ref{station-least-en}, given $2<p<6$, $\alpha\in \R\setminus \{0\}$ and $\mu>0$ such that $\A_{\mu}\neq\emptyset$, then there exists a unique positive bound state $\eta^{\omega}$ with least energy in $\A_{\mu}$. In the following, we denote by $\eta^{\mu}$ the least energy bound state of mass $\mu$. In particular, if $4<p<6$ and $\alpha>0$, then $\eta^{\mu}=\eta^{\omega}$, with $\omega\geq \omega^{*}(\alpha)$.
 \end{remark}
 
\subsection{Bound states in the critical case}

The next proposition and corollary collect some properties of the function $M$ and of the set $\A_{\mu}$ defined in \eqref{defmu} and \eqref{A-mu} respectively when $p=6$.
\begin{proposition}
\label{mu-study-crit}
Let $p=6$ and $\alpha\in \R\setminus\{0\}$. Then the function $M$ defined in \eqref{defmu} is of class $C^{1}((\alpha^{2},+\infty))$ and
\begin{equation}
\label{mass-eta2}
M(\omega)=\frac{\sqrt{3}}{2}\left(\f{\pi}{2}+\arcsin\left(\f{\alpha}{\sqrt{\omega}}\right)\right).
\end{equation}
Moreover:
\begin{itemize}
\item[$(i)$]if $\alpha<0$, then $M'(\omega)>0$ for every $\omega\in(\alpha^{2},+\infty)$ and $M((\alpha^{2},+\infty))=\left(0,\f{\sqrt{3}\pi}{4}\right)$,
\item[$(ii)$] if $\alpha>0$, then $M'(\omega)<0$ for every $\omega\in(\alpha^{2},+\infty)$ and $M((\alpha^{2},+\infty))=\left(\f{\sqrt{3}\pi}{4},\f{\sqrt{3}\pi}{2}\right)$.
\end{itemize}
\end{proposition}
\begin{proof}
The expression \eqref{mass-eta2} and the limits at the endpoints of the domain are obtained by straightforward computations. Moreover, the derivative looks like
\begin{equation*}
M'(\omega)=-\alpha \f{\sqrt{3}}{4}\f{1}{\omega\sqrt{\omega-\alpha^{2}}},
\end{equation*}
which is strictly positive or strictly negative if $\alpha<0$ or $\alpha>0$ respectively, entailing the thesis.
\end{proof}
\begin{corollary}
\label{station-numb-crit}
Let $\alpha\in \R\setminus\{0\}$ and $p=6$. Therefore:
\begin{itemize}
\item[$(i)$]if $\alpha<0$, then 
\[
|\A_{\mu}|=
\begin{cases}
1 \quad\text{if}\quad \mu\in \left(0,\f{\sqrt{3}\pi}{4}\right),\\
0 \quad\text{if}\quad \mu\in\left[\f{\sqrt{3}\pi}{4},+\infty\right),
\end{cases}
\] 
\item[$(ii)$] if $\alpha>0$, then 
\begin{equation*}
|\A_{\mu}|=
\begin{cases}
0\quad \text{if}\quad \mu\in \left(0,\f{\sqrt{3}\pi}{4}\right]\cup\left[\f{\sqrt{3}\pi}{2},+\infty\right),\\
1\quad \text{if}\quad \mu\in \left(\f{\sqrt{3}\pi}{4}, \f{\sqrt{3}\pi}{2}\right) ;
\end{cases}
\end{equation*}
\end{itemize}
\end{corollary}
\begin{proof}
The proof of $(i)$ and $(ii)$ is a straightforward consequence of $(i)$ and $(ii)$ of Proposition \ref{mu-study-crit}.
\end{proof}

\section{Proof of Theorems \ref{sigmaneg-p<6}, \ref{sigmapos-pleq4}, \ref{sigmapos-p>4} and Proposition \ref{GSalpha}: the subcritical case}

\label{thm-subcase}
The next proposition and corollary provide an existence criterion allowing us to reduce the problem of the existence of ground states to a comparison between the energy of the lowest energy bound state and the standard energy of the soliton on the line. {Analogous results have been obtained in the context of metric graphs with Kirchhoff conditions at the vertices in \cite{AST}: our proofs are just a minor modification of the proofs in \cite{AST}, but we report them here for the sake of completeness.}

\begin{proposition}
\label{compactth}
Let $2<p<6$ and $\alpha\in \R\setminus\{0\}$ . Then, for every $\mu>0$ it holds
\begin{equation}
\label{F leq E}
\F(\mu)\leq\Eps(\mu,\R)\,.
\end{equation}
Furthermore, if $\F(\mu)<\Eps(\mu,\R)$, then ground states of \eqref{FR+} at mass $\mu$ exist.
\end{proposition}
\begin{proof}
{ We first observe that \eqref{F leq E} has been proved in a more general context in \cite[Proposition A.1]{C-18}}. 
%Let us first prove \eqref{F leq E}. For every $\varepsilon>0$, let $v_\varepsilon:=\kappa_\varepsilon(\phi_\mu-\varepsilon)_+$, where $\phi_\mu$ is the soliton of mass $\mu$ and $\kappa_\varepsilon>0$ is chosen to guarantee $v_\varepsilon\in H_\mu^1(\R)$. Since $\|v_\varepsilon\|_{L^q(\R)}\to\|\phi_\mu\|_{L^q(\R)}$ as $\varepsilon\to0$, for every $q\geq1$, then $\kappa_\varepsilon\to1$ for $\varepsilon\to0$, and we get 
%\[
%\mathcal{E}(\mu,\R)\leq E(v_\varepsilon,\R)=\f12\kappa_\varepsilon^2\int_\R|(\phi_\mu-\varepsilon)_+'|^2\,dx-\f1p\kappa_\varepsilon^{ p}\int_\R|(\phi_\mu-\varepsilon)_+|^p\,dx\leq \mathcal{E}(\mu,\R)+o(1)
%\]
%for $\varepsilon$ small enough, making use also of $\|v_\varepsilon'\|_{L^2(\R)}\leq\|\phi_{\mu}'\|_{L^2(\R)}$. Hence, $E(v_\varepsilon,\R)\to\mathcal{E}(\mu,\R)$ as $\varepsilon\to0$. Moreover, $v_\varepsilon$ has compact support, so that one can think of it as supported on $\R^{+}$. We thus have
%\[
%\Eps(\mu,\R)=\lim_{\varepsilon\to0^+}E(v_\varepsilon,\R)=\lim_{\varepsilon\to0^+}F(v_\varepsilon)\geq\F(\mu)\,,
%\]
%so \eqref{F leq E} is proved.
Assume now that $\F(\mu)<\Eps(\mu,\R)$ and let $(u_n)\subset H^1_\mu(\R^{+})$ be a
minimizing sequence for \eqref{FR+}. Plugging \eqref{GNp} and \eqref{GNinf} into the definition of $F$ gives
\begin{equation*}
F(u_n)\ge\frac{1}{2}\|u_n'\|_2^2-\frac{K_p}{p}\mu^\frac{p+2}{4}\|u_n'\|_2^{\frac{p}{2} -1}-|\alpha|\mu^\frac{1}{2}\|u_n'\|_2
\end{equation*}
which ensures that $(u_n)$ is bounded in $H^1(\R^{+})$ since $2<p<6$.
Therefore there exists $u\in H^1(\R^{+})$ such that, up to subsequences, $u_n \deb u$ weakly in $H^1(\R^{+})$, $u_n\rightarrow u$ in $L^\infty_{loc}(\R^{+})$ and consequently $u_n\rightarrow u$ a.e. in $\R^{+}$.

Set $m:=\|u\|_2^2$. By weak lower semicontinuity, we have $m\le \mu$.

Assume $m=0$, that is $u\equiv 0$. Then $u_n(0)\to0$ as $n\to+\infty$, so that, if we define 
\begin{equation*}
\overline{u_{n}}(x):=
\begin{cases}
0 \quad &\text{if}\quad x\leq -u_{n}(0),\\
x+u_{n}(0) \quad &\text{if}\quad -u_{n}(0)<x<0,\\
u_{n}(x)\quad &\text{if}\quad x\geq 0,
\end{cases}
\end{equation*} 
then
\begin{equation*}
\Eps(\mu,\R)>\F(\mu)=\lim_n F(u_n)=\lim_n E(\overline{u_n},\R)\geq\lim_{n}\Eps\left(\mu+\f{u_{n}(0)^{3}}{3},\R\right)\geq\Eps(\mu,\R),
\end{equation*}
i.e., a contradiction. Hence, $u\not\equiv0$ on $\R^{+}$.

Suppose then that $0<m<\mu$. By weak convergence in $H^1(\R^{+})$ of $u_n$ to $u$, we get $\|u_n-u\|_2^2=\mu-m+o(1)$ for $n\to+\infty$. On the one hand, since $p>2$ and $\frac{\mu}{\|u_n-u\|_2^2}>1$ for $n$ sufficiently large,
\begin{equation*}
\begin{split}
&\F(\mu)\le F\left(\sqrt{\frac{\mu}{\|u_n-u\|_2^2}}(u_n-u)\right)\\
&=\f12\f{\mu}{\|u_n-u\|_2^2}\|u'_n-u'_n\|_2^2-\frac{1}{p}\left(\frac{\mu}{\|u_n-u\|_2^2}\right)^{\frac{p}{2}}\|u_n-u\|_p^p\\
&-\frac{1}{2}\frac{\mu}{\|u_n-u\|_2^2}\abs{u_n(0)-u(0)}^2<\frac{\mu}{\|u_n-u\|_2^2}F(u_n-u,\R^{+}), 
\end{split}
\end{equation*}
so that
\begin{equation}
\label{Fun-est}
\liminf_n F(u_n-u)\ge \frac{\mu-m}{\mu}\F(\mu).
\end{equation}
On the other hand, an analogous reasoning leads to
\begin{equation*}
\F(\mu)\le F\left(\sqrt{\frac{\mu}{\|u\|_2^2}}\,u\right)<\frac{\mu}{\|u\|_2^2}F(u),
\end{equation*}
so
\begin{equation}
\label{Fu-est}
F(u)>\frac{m}{\mu}\F(\mu).
\end{equation}
Moreover, it holds
\begin{equation}
\label{FpqBrezLieb}
F(u_n)=F(u_n-u)+F(u)+o(1).
\end{equation}
Indeed, by $u'_n\deb u'$ weakly in $L^2(\R^{+})$ and $u_n\to u$ in $L_{\text{loc}}^\infty(\R^{+})$, we have
$\|u'_n-u'\|_2^2=\|u'_n\|_2^2-\|u'\|_2^2+o(1)$ and $|(u_n-u)(0)|^2=|u_{n}(0)|^{2}-|u(0)|^{2}+o(1)$ as $n$ is large enough. Furthermore, owing to the Brezis-Lieb lemma \cite{BL},
\[
\|u_n\|_p^p=\|u_n-u\|_p^p+\|u\|_p^p+o(1).
\]
Using now \eqref{Fun-est}, \eqref{Fu-est} and \eqref{FpqBrezLieb}, we get
\begin{equation*}
\begin{split}
\F(\mu)&=\lim_n F(u_n)=\lim_n F(u_n-u)+F(u)\\
&>\frac{\mu-m}{\mu}\F(\mu)+\frac{m}{\mu}\F(\mu)=\F(\mu),
\end{split}
\end{equation*}
which is again a contradiction.

Henceforth, $m=\mu$ and $u\in H^{1}_{\mu}(\R^{+})$. In particular, $u_n\rightarrow u$ in $L^2(\R^{+})$ so that, $(u_n)$ being bounded in $L^\infty(\R^{+})$, $u_n\rightarrow u$ in $L^p(\R^{+})$ as $n\to+\infty$. Thus, by weak lower semicontinuity
\[
F(u)\leq\lim_n F(u_n)=\F(\mu)\,,
\]
that is $u$ is a ground state of \eqref{FR+} at mass $\mu$.
\end{proof}
\begin{corollary}
\label{compactcor}
Let $2<p<6$, $\alpha\in \R\setminus\{0\}$ and $\mu>0$ be fixed. {Then ground states of \eqref{FR+} at mass $\mu$ exist if and only if there exists $u\in H^1_\mu(\R^{+})$ such that $F(u)\le \Eps(\mu,\R)$}.
\end{corollary}
\begin{proof}
Suppose first that there exists $u\in H^1_\mu(\R^{+})$ such that $F(u)\le \Eps(\mu,\R)$. If $\F(\mu)=F(u)$ then $u$ is a ground state of \eqref{FR+} at mass $\mu$. Otherwise, $\F(\mu)<F(u)\leq\Eps(\mu,\R)$ and a ground state of \eqref{FR+} at mass $\mu$ exists by Proposition \ref{compactth}. 

{The other implication follows from the definition of ground state and the first part of Proposition \ref{compactth}.}
\end{proof}
We are now ready to prove Theorem \ref{sigmaneg-p<6}, Theorem \ref{sigmapos-pleq4}, Theorem \ref{sigmapos-p>4} and Proposition \ref{GSalpha}.

\begin{proof}[Proof of Theorem \ref{sigmaneg-p<6}]
Let $u=\phi_{2\mu}\mathbb{1}_{\R^{+}}$ be the half-soliton of mass $\mu$. Then by \eqref{E mu R+}
\[
F(u)=\Eps(\mu,\R^{+})+\f{\alpha}{2}|u(0)|^{2}<\Eps(\mu,\R^{+})<\Eps(\mu,\R),
\]
hence by Corollary \ref{compactcor} there exists a ground state of \eqref{FR+} at mass $\mu$. The ground state is unique since every ground state belongs to $\A_{\mu}$ and by $(i)$ of Corollary \ref{station-numb} the set $\A_{\mu}$ has cardinality one when $\alpha<0$.
\end{proof}

\begin{proof}[Proof of Theorem \ref{sigmapos-pleq4}]
If $0<\mu\leq \|\phi_{\alpha^{2}}\|_{L^{2}(\R)}^{2}$, then by $(ii)$ of Corollary \ref{station-numb} the set $\A_{\mu}$ is empty, hence ground states at mass $\mu$ do not exist. 

Suppose now that $\mu>\|\phi_{\alpha^{2}}\|_{L^{2}(\R)}^{2}$. If there exists a ground state, then by Remark \ref{unique-stat} it is unique and it coincides with $\eta^{\mu}:=\eta^{\omega(\mu)}=\phi_{\omega}(\cdot-a)$. Therefore, relying on Corollary \ref{compactcor}, we have that ground states exist if and only if 
\begin{equation*}
F\left(\eta^{\omega(\mu)}\right)\le \Eps(\mu,\R),
\end{equation*}
which by \eqref{E phi mu} can be rewritten as
\begin{equation*}
\f{F\left(\eta^{\omega(\mu)}\right)}{\mu^{2\beta+1}}\le -\theta_{p}.
\end{equation*}
Set $K(\mu):=\f{F\left(\eta^{\omega(\mu)}\right)}{\mu^{2\beta+1}}$ for every $\mu>\|\phi_{\alpha^{2}}\|_{L^{2}(\R)}^{2}$. 

By using \eqref{mass-eta}, we have that $\omega(\mu)\to \alpha^{2}$, $\eta^\mu(0)\to 0$ and $a\to+\infty$ as $\mu\to \|\phi_{\alpha^{2}}\|_{L^{2}(\R)}^{2}$. This means that $F\left(\eta^{\omega(\mu)}\right)\to \Eps\left(\|\phi_{\alpha^{2}}\|_{L^{2}(\R)}^{2},\R\right)$ as $\mu\to \|\phi_{\alpha^{2}}\|_{L^{2}(\R)}^{2}$ and, as a consequence, $K(\mu)\to -\theta_{p}$ as $\mu\to \|\phi_{\alpha^{2}}\|_{L^{2}(\R)}^{2}$. 

Moreover, since $M'(\omega)>0$ for every $\omega>\alpha^{2}$ by Proposition \ref{mu-study}, it turns out that $F(\eta^{\mu})$ is differentiable with respect to $\mu$ and, for every $\overline{\mu}>\|\phi_{\alpha^{2}}\|_{L^{2}(\R)}^{2}$, there exists an only value $\overline{\omega}>\alpha^{2}$ such that $M(\overline{\omega})=\overline{\mu}$ and 
\[
\f{d F\left(\eta^{\mu}\right)}{d\mu}_{\Big|\mu=\overline{\mu}}=\f{d F(\eta^{\omega})}{d\omega}_{\Big|\omega=\overline{\omega}}\f{d\omega}{d\mu}(\overline{\mu})=\f{\f{d F(\eta^{\omega})}{d\omega}_{|\omega=\overline{\omega}}}{M'(\overline{\omega})}.
\]
Therefore,
\begin{equation}
\label{K prime}
K'(\overline{\mu})=\f{1}{\overline{\mu}^{2\beta+1}}\left[\f{\f{d F(\eta^{\omega})}{d\omega}_{|\omega=\overline{\omega}}}{M'(\overline{\omega})}-(2\beta+1)\f{F\left(\eta^{\overline{\mu}}\right)}{\overline{\mu}}\right],
\end{equation}
and, by substituting \eqref{mu prime}, \eqref{F-omega} and \eqref{dF domega} in \eqref{K prime}, there results
\begin{equation*}
K'(\mu)=-\alpha\f{p-2}{6-p}\f{|\eta^{\mu}(0)|^{2}}{\mu^{2\beta+2}}<0\quad \forall\,\mu>\|\phi_{\alpha^{2}}\|_{L^{2}(\R)}^{2}.
\end{equation*}

Since the function $K$ is strictly decreasing in $\mu$ and $K(\mu)\to -\theta_{p}$ as $\mu\to \|\phi_{\alpha^{2}}\|_{L^{2}(\R)}^{2}$, we can conclude that $F(\eta^{\mu})<\Eps(\mu,\R)$ for every $\mu>\|\phi_{\alpha^{2}}\|_{L^{2}(\R)}^{2}$, hence by Corollary \ref{compactcor} ground states of \eqref{FR+} at mass $\mu$ exist.
\end{proof}

\begin{proof}[Proof of Theorem \ref{sigmapos-p>4}]
If $0<\mu< \mu^{*}$, then by $(iii)$ of Corollary \ref{station-numb} the set $\A_{\mu}$ is empty, hence ground states at mass $\mu$ do not exist. 

Suppose now that $\mu\geq \mu^{*}$. If there exists a ground state, then by Remark \ref{unique-stat} it is unique and, by $(iii)$ of Corollary \ref{station-numb} and Proposition \ref{station-least-en}, it coincides with the only positive bound state $\eta^\mu=\eta^{\omega}$ with $\omega\geq\omega^{*}$.
By relying on Corollary \ref{compactcor}, one can deduce that ground states exist if and only if
\begin{equation*}
K(\mu):=\f{F\left(\eta^{\mu}\right)}{\mu^{2\beta+1}}\leq -\theta_{p}.
\end{equation*}

If $\mu^{*}<\mu<\|\phi_{\alpha^{2}}\|_{L^{2}(\R)}^{2}$, then by Remark \ref{2-stat} and Proposition \ref{station-least-en}  there exists $\alpha^{2}<\omega_{1}< \omega^{*}$ such that $\eta^{\omega_{1}}$ is not the least energy positive bound state of mass $\mu$. In this regard, the function $K_{1}(\mu):=\f{F\left(\eta^{\omega_{1}}\right)}{\mu^{2\beta+1}}$ is continuous in $\left(\mu^{*},\|\phi_{\alpha^{2}}\|_{L^{2}(\R)}^{2}\right)$, 
\[
\lim_{\mu\to(\mu^{*})^{+}}K_{1}(\mu)=\f{F\left(\eta^{\mu^{*}}\right)}{\left(\mu^{*}\right)^{2\beta+1}}=K(\mu^{*})
\] 
and, using the same arguments adopted in the proof of Theorem \ref{sigmapos-pleq4},
\[
\lim_{\mu\to\left(\|\phi_{\alpha^{2}}\|_{L^{2}(\R)}^{2}\right)^{-}} K_{1}(\mu)=-\theta_{p}.
\]
Moreover, $K_{1}$ is differentiable in $\left(\mu^{*},\|\phi_{\alpha^{2}}\|_{L^{2}(\R)}^{2}\right)$ and, repeating the computations done in Theorem \ref{sigmapos-pleq4} for $K$, there results $K_{1}'<0$ in $\left(\mu^{*},\|\phi_{\alpha^{2}}\|_{L^{2}(\R)}^{2}\right)$, entailing that
\begin{equation*}
K(\mu^{*})=\lim_{\mu\to(\mu^{*})^{+}}K_{1}(\mu)>-\theta_{p}.
\end{equation*}

{Furthermore, $\omega(\mu)\to +\infty$ and $a\to 0$ as $\mu\to +\infty$. More specifically, by combining  \eqref{mass-eta}, \eqref{mass phi w} and \eqref{thetap}, we get that 
\begin{equation}
\label{om-mu-inf}
\omega=2^{2\beta+1}\theta_{p}(2\beta+1)\mu^{2\beta}+o(\mu^{2\beta}),\quad \text{as}\quad \mu\to+\infty.
\end{equation} 

By applying \eqref{om-mu-inf} to \eqref{F-omega}, there results
\begin{equation*}
\begin{split}
F(\eta^{\mu})&= -\theta_{p}2^{2\beta}\mu^{2\beta+1}+o(\mu^{2\beta+1})+\f{\alpha}{2}C_{p}\mu^{\f{4}{6-p}}+o\left(\mu^{\f{4}{6-p}}\right)\\
&=-\theta_{p}2^{2\beta}\mu^{2\beta+1}+o(\mu^{2\beta+1}) \quad\text{as}\quad\mu\to +\infty,
\end{split}
\end{equation*}
entailing that $K(\mu)\to -2^{2\beta}\theta_{p}<-\theta_{p}$ when $\mu \to+\infty$.

Since $K$ is a continuous and strictly decreasing function in $(\mu^{*},+\infty)$, $K(\mu^{*})>-\theta_{p}$ and $\lim_{\mu\to+\infty}K(\mu)<-\theta_{p}$, there results} that there exists a unique $\widetilde{\mu}>\mu^{*}$ such that $K(\widetilde{\mu})= -\theta_{p}$ and $K(\mu)<-\theta_{p}$ if and only if $\mu>\widetilde{\mu}$. In order to prove the upper bound for $\widetilde{\mu}$, one relies on Proposition \ref{station-least-en} and on  \eqref{dF domega}, getting that
\[
K\left(\|\phi_{\alpha^{2}}\|_{L^{2}(\R)}^{2}\right)=\lim_{\mu\to \left(\|\phi_{\alpha^{2}}\|_{L^{2}(\R)}^{2}\right)^{-}} K(\mu)< \lim_{\mu\to \left(\|\phi_{\alpha^{2}}\|_{L^{2}(\R)}^{2}\right)^{-}} K_{1}(\mu)=-\theta_{p}, 
\]
so that $\widetilde{\mu}<\|\phi_{\alpha^{2}}\|_{L^{2}(\R)}^{2}$ and the thesis follows. 
\end{proof}

\begin{proof}[Proof of Proposition \ref{GSalpha}]
Fix $\mu>0$ and $2<p<6$. We preliminary observe that ground states of mass $\mu$ exists for every $\alpha\leq 0$: indeed, if $\alpha=0$, then the only ground state coincides with half of the soliton of mass $2\mu$, as pointed out in Subsection \ref{subRR+}, while if $\alpha<0$, then ground states exist by Theorem \ref{sigmaneg-p<6}. In order to deduce for which $\alpha>0$ ground states exist, let us distinguish the cases $2<p\leq 4$ and $4<p<6$. If $2<p\leq 4$, then by Theorem \ref{sigmapos-pleq4} ground states at mass $\mu$ exists if and only if $\mu>\|\phi_{\alpha^{2}}\|_{L^{2}(\R)}^{2}$, i.e. if and only if
\begin{equation*}
\mu>\f{4\left(\f{p}{2}\right)^{\f{2}{p-2}}\alpha^{\f{6-p}{p-2}}}{p-2}\int_{0}^{1}(1-s^{2})^{\f{4-p}{p-2}}\,ds,
\end{equation*}
that entails $(i)$.

If instead $4<p<6$, then by Theorem \ref{sigmapos-p>4} ground states at mass $\mu$ exist if and only if $\mu\geq \widetilde{\mu}=\widetilde{\mu}(\alpha)$, with $\mu^{*}(\alpha)<\widetilde{\mu}(\alpha)<\|\phi_{\alpha^{2}}\|_{L^{2}(\R)}^{2}$. In particular, if we denote by $\eta^{\widetilde{\mu}}=\eta^{\widetilde{\omega}}$ the only ground state at mass $\widetilde{\mu}$ or, alternatively, at frequency $\widetilde{\omega}$, with $\widetilde{\omega}>\omega^{*}$, then it satisfies $\widetilde{\mu}=M(\widetilde{\omega},\alpha)$ and  $F\left(\eta^{\widetilde{\omega}}\right)=-\theta_{p}\widetilde{\mu}^{2\beta+1}$. In particular, in view of  Proposition \ref{mu-study} the condition $\widetilde{\omega}>\omega^{*}$ reduces to the equation $\f{6-p}{2}\widetilde{\mu}>\left(\f{p}{2}\right)^{\f{2}{p-2}}\alpha(\widetilde{\omega}-\alpha^{2})^{\f{4-p}{p-2}}$.  In order to invert the inequality $\mu\geq \widetilde{\mu}(\alpha)$, we need to investigate if $\widetilde{\mu}$ is invertible as a function of $\alpha$. In {order to do this}, let us notice that the triple $(\widetilde{\mu},\widetilde{\omega},\alpha)$ satisfies the system 
\begin{equation}
\label{eq-mu-om-al}
\begin{cases}
\mu=M(\omega,\alpha),\\
F(\eta^{\mu})=-\theta_{p}\mu^{2\beta+1},
\end{cases}
\end{equation}
and the additional constraints
\begin{equation}
\label{constraint-mu}
\begin{cases}
\alpha>0,\\
\omega>\alpha^{2},\\
\f{6-p}{2}\mu>\left(\f{p}{2}\right)^{\f{2}{p-2}}\alpha(\omega-\alpha^{2})^{\f{4-p}{p-2}}.
\end{cases}
\end{equation}
{Let us observe that \eqref{eq-mu-om-al} follows from imposing the mass constraint and the fact that the energy $F(\eta^{\mu})$ at mass $\widetilde{\mu}$ equals the energy of the soliton on the real line of mass $\widetilde{\mu}$, while inequalities \eqref{constraint-mu} encode the condition $\widetilde{\omega}>\omega^{*}$}.

Therefore, in view of \eqref{mass-eta} and \eqref{F-omega} it is possible to rewrite the system \eqref{eq-mu-om-al} as 
\begin{equation*}
G(\mu,\omega,\alpha)=\begin{pmatrix}
G_{1}(\mu,\omega,\alpha)\\
G_{2}(\mu,\omega,\alpha)
\end{pmatrix}
=\begin{pmatrix}
0\\
0
\end{pmatrix},
\end{equation*}
where
\begin{equation*}
G_{1}(\mu,\omega,\alpha):=\frac{p^\frac{2}{p-2}}{2^{\f{4-p}{p-2}}(p-2)}\omega^{\frac{6-p}{2(p-2)}}\int_{-\f{\alpha}{\sqrt{\omega}}}^1(1-s^2)^{\frac{4-p}{p-2}}\,ds-\mu
\end{equation*}
and
\begin{equation*}
G_{2}(\mu,\omega,\alpha):=-\f{6-p}{2(p+2)}\omega\mu+\f{\alpha (p-2)}{2(p+2)}\left(\f{p}{2}\right)^{\f{2}{p-2}}(\omega-\alpha^{2})^{\f{2}{p-2}}+\theta_{p}\mu^{2\beta+1}.
\end{equation*}
By direct computations {and using \eqref{eq-mu-om-al}, one gets}
\begin{equation*}
\begin{cases}
\f{\de G_{1}}{\de\mu}(\widetilde{\mu},\widetilde{\omega},\alpha)&=-1\\
\f{\de G_{1}}{\de\omega}(\widetilde{\mu},\widetilde{\omega},\alpha)&=\f{1}{(p-2)\widetilde{\omega}}\left(\f{6-p}{2}\widetilde{\mu}-\left(\f{p}{2}\right)^{\f{2}{p-2}}\alpha(\widetilde{\omega}-\alpha^{2})^{\f{4-p}{p-2}}\right)\\
\f{\de G_{1}}{\de\alpha}(\widetilde{\mu},\widetilde{\omega},\alpha)&=\f{2}{p-2}\left(\f{p}{2}\right)^{\f{2}{p-2}}(\widetilde{\omega}-\alpha^{2})^{\f{4-p}{p-2}}\\
\f{\de G_{2}}{\de\mu}(\widetilde{\mu},\widetilde{\omega},\alpha)&=-\f{6-p}{2(p+2)}\widetilde{\omega}+(2\beta+1)\theta_{p}\widetilde{\mu}^{2\beta}\\
\f{\de G_{2}}{\de\omega}(\widetilde{\mu},\widetilde{\omega},\alpha)&=-\f{6-p}{2(p+2)}\widetilde{\mu}+\f{\alpha}{p+2}\left(\f{p}{2}\right)^{\f{2}{p-2}}(\widetilde{\omega}-\alpha^{2})^{\f{4-p}{p-2}}\\
\f{\de G_{2}}{\de\alpha}(\widetilde{\mu},\widetilde{\omega},\alpha)&=\f{1}{p+2}\left(\f{p}{2}\right)^{\f{2}{p-2}}(\widetilde{\omega}-\alpha^{2})^{\f{4-p}{p-2}}\left(\f{p-2}{2}\widetilde{\omega}-\f{p+2}{2}\alpha^{2}\right),
\end{cases}
\end{equation*}
hence 
\begin{equation}
\label{det}
\text{det}\left(\begin{pmatrix}
\f{\de G_{1}}{\de\mu} & \f{\de G_{1}}{\de\omega}\\
\f{\de G_{2}}{\de\mu} & \f{\de G_{2}}{\de\omega}
\end{pmatrix}(\widetilde{\mu},\widetilde{\omega},\alpha)\right)
={\f{\de G_{1}}{\de\omega}(\widetilde{\mu},\widetilde{\omega},\alpha)}\left(\f{\widetilde{\omega}}{2}-\f{p+2}{6-p}\theta_{p}\widetilde{\mu}^{2\beta}\right).
\end{equation}
We observe that ${\f{\de G_{1}}{\de\omega}(\widetilde{\mu},\widetilde{\omega},\alpha)}>0$ since $\f{6-p}{2}\widetilde{\mu}>\left(\f{p}{2}\right)^{\f{2}{p-2}}\alpha(\widetilde{\omega}-\alpha^{2})^{\f{4-p}{p-2}}$ {by \eqref{constraint-mu}} and, by using $G_{2}(\widetilde{\mu},\widetilde{\omega},\alpha)=0$, there results
\begin{equation*}
\f{\widetilde{\omega}}{2}-\f{p+2}{6-p}\theta_{p}\widetilde{\mu}^{2\beta}=\f{{\alpha}(p-2)}{2(6-p)\widetilde{\mu}}\left(\f{p}{2}\right)^{\f{2}{p-2}}(\widetilde{\omega}-\alpha^{2})^{\f{2}{p-2}}>0,
\end{equation*}
Thus, {since \eqref{det} is positive}, the Implicit function theorem applies and
\begin{equation*}
\begin{pmatrix}
\widetilde{\mu}'(\alpha)\\
\widetilde{\omega}'(\alpha)
\end{pmatrix}
=-\begin{pmatrix}
\f{\de G_{1}}{\de\mu}(\widetilde{\mu},\widetilde{\omega},\alpha) & \f{\de G_{1}}{\de\omega}(\widetilde{\mu},\widetilde{\omega},\alpha)\\
\f{\de G_{2}}{\de\mu}(\widetilde{\mu},\widetilde{\omega},\alpha) & \f{\de G_{2}}{\de\omega}(\widetilde{\mu},\widetilde{\omega},\alpha)
\end{pmatrix}^{-1}
\begin{pmatrix}
\f{\de G_{1}}{\de\alpha}(\widetilde{\mu},\widetilde{\omega},\alpha)\\
\f{\de G_{2}}{\de\alpha}(\widetilde{\mu},\widetilde{\omega},\alpha)
\end{pmatrix},
\end{equation*}
with $\widetilde{\mu}=\widetilde{\mu}(\alpha)$ and $\widetilde{\omega}=\widetilde{\omega}(\alpha)$.
In particular,
\begin{equation*}
\widetilde{\mu}'(\alpha)=-\f{\f{\de G_{2}}{\de\omega}(\widetilde{\mu},\widetilde{\omega},\alpha)
\f{\de G_{1}}{\de\alpha}(\widetilde{\mu},\widetilde{\omega},\alpha)-\f{\de G_{1}}{\de\omega}(\widetilde{\mu},\widetilde{\omega},\alpha)\f{\de G_{2}}{\de\alpha}(\widetilde{\mu},\widetilde{\omega},\alpha)}{\text{det} \begin{pmatrix}
\f{\de G_{1}}{\de\mu}(\widetilde{\mu},\widetilde{\omega},\alpha) & \f{\de G_{1}}{\de\omega}(\widetilde{\mu},\widetilde{\omega},\alpha)\\
\f{\de G_{2}}{\de\mu}(\widetilde{\mu},\widetilde{\omega},\alpha) & \f{\de G_{2}}{\de\omega}(\widetilde{\mu},\widetilde{\omega},\alpha)
\end{pmatrix}
}.
\end{equation*}
Since the denominator is positive and the numerator
\begin{equation*}
\begin{split}
&\f{\de G_{2}}{\de\omega}(\widetilde{\mu},\widetilde{\omega},\alpha)
\f{\de G_{1}}{\de\alpha}(\widetilde{\mu},\widetilde{\omega},\alpha)-\f{\de G_{1}}{\de\omega}(\widetilde{\mu},\widetilde{\omega},\alpha)\f{\de G_{2}}{\de\alpha}(\widetilde{\mu},\widetilde{\omega},\alpha)\\
&=-\left(\f{p}{2}\right)^{\f{2}{p-2}}\f{(\widetilde{\omega}-\alpha^{2})^{\f{2}{p-2}}}{2(p-2)\widetilde{\omega}}\left(\f{6-p}{2}\widetilde{\mu}-\left(\f{p}{2}\right)^{\f{2}{p-2}}\alpha(\widetilde{\omega}-\alpha^{2})^{\f{4-p}{p-2}}\right)<0,
\end{split}
\end{equation*}
there results that $\widetilde{\mu}'(\alpha)>0$, hence $\widetilde{\mu}$ is a strictly increasing function of $\alpha$ and the thesis follows.
\end{proof}
\section{Proof of Theorems \ref{sigmaneg-p6}, \ref{sigmapos-p6}: the critical case}

\label{thm-critcase}
In this section we prove Theorem \ref{sigmaneg-p6} and \ref{sigmapos-p6}.
\begin{proof}[Proof of Theorem \ref{sigmaneg-p6}]
{ First, let us observe that by \cite[Theorem 2]{C-18} and the fact that $K_{6}(\R^{+})=\f{16}{\pi^{2}}$, we  deduce that ground states exist for every $0<\mu<\f{\sqrt{3}}{K_{6}(\Rp)}=\f{\sqrt{3}\pi}{4}$. Moreover, since by Corollary \ref{station-numb-crit} there exists only one bound states for $0<\mu<\f{\sqrt{3}}{4}$, then the ground state is unique and coincide with the only bound state.}

{ Second}, let us consider $v=\phi_\omega\mathbb{1}_{\R^{+}}$, with $\phi_{\omega}$ as in \eqref{sol-p6} and satisfying $\|v\|_2^{2}=\f{\sqrt{3}\pi}{4}$ and $E(v,\R^{+})=0$ for every $\omega\in \R$, and define $v_{\mu}:=\la_{\mu}v$, with $\la_{\mu}>0$ such that $v_{\mu}\in H^{1}_{\mu}(\R^{+})$. Therefore, for every $\mu\geq \f{\sqrt{3}\pi}{4}$, there results that $\la_{\mu}\geq1$ and
\begin{equation*}
\begin{split}
F(v_{\mu})&=\f{\la_{\mu}^{2}}{2}\|v'\|_2^{2}-\f{\la_{\mu}^{6}}{6}\|v\|_6^{6}+\f{\la_{\mu}^{2}\alpha}{2}|v(0)|^{2}\\
&\leq \la_{\mu}^{2}\left(E(v,\R^{+})+\f{\alpha}{2}|v(0)|^{2}\right)=\f{\la_{\mu}^{2}\alpha}{2}|v(0)|^{2}<0.
\end{split}
\end{equation*}

As a consequence, by applying the mass preserving transformation $f\mapsto f_{\nu}:=\sqrt{\nu}f(\nu\cdot)$ to $v_{\mu}$, we get
\[
F((v_{\mu})_{\nu})=\nu^{2}E(v_{\mu},\R^{+})-\f{\nu|\alpha|}{2}|v_{\mu}|^{2}\leq\nu F(v_{\mu})\to -\infty\quad \text{as}\quad \nu \to +\infty,
\]
hence $\F(\mu)=-\infty$ if $\mu\geq\f{\sqrt{3}\pi}{4}$.

\end{proof}
\begin{proof}[Proof of Theorem \ref{sigmapos-p6}]
Arguing similarly as in Theorem \ref{sigmaneg-p6}, we consider $v=\phi_\omega\mathbb{1}_{\R^{+}}$, with $\phi_{\omega}$ as in \eqref{sol-p6} and satisfying $\|v\|_2^{2}=\f{\sqrt{3}\pi}{4}$ and $E(v,\R^{+})=0$ for every $\omega>0$ and we define $v_{\mu}:=\la_{\mu}v$, with $\la_{\mu}>0$ such that $v_{\mu}\in H^{1}_{\mu}(\R^{+})$: we highlight that $v$ and $v_{\mu}$ depend on $\omega$, but we have omitted this dependence to simplify the notation. Therefore, for every $\mu>\f{\sqrt{3}\pi}{4}$, there results that $\la_{\mu}>1$ and
\begin{equation*}
\begin{split}
F(v_{\mu})&=\f{\la_{\mu}^{2}}{2}\|v'\|_2^{2}-\f{\la_{\mu}^{6}}{6}\|v\|_6^{6}+\f{\la_{\mu}^{2}\alpha}{2}|v(0)|^{2}\\
&=\la_{\mu}^{2}E(v,\R^{+})+\la_{\mu}^{2}\left(\f{\alpha}{2}|v(0)|^{2}-(\la_{\mu}^{4}-1)\|v\|_6^{6}\right)=\\
&=\la_{\mu}^{2}\left(\f{\alpha}{2}|v(0)|^{2}-(\la_{\mu}^{4}-1)\|v\|_6^{6}\right)
=\f{\la_{\mu}^{2}\sqrt{3\omega}}{2}\left(\alpha-\f{\pi}{8}(\la_{\mu}^{4}-1)\sqrt{\omega}\right).
\end{split}
\end{equation*}
In particular, if we choose $\omega>\left(\f{8\alpha}{\pi(\la_{\mu}^{4}-1)}\right)^{2}$, then $F(v_{\mu})<0$.  As done for Theorem \ref{sigmaneg-p6}, one defines $(v_{\mu})_{\nu}(x)=\sqrt{\nu}v_{\mu}(\nu x)$, so that $F\left((v_{\mu})_{\nu}\right)<\nu F(v_{\mu})\to -\infty$ as $\nu\to +\infty$, hence $\F(\mu)=-\infty$ for $\mu>\f{\sqrt{3}\pi}{4}$.

On the contrary, if $\mu\leq\f{\sqrt{3}\pi}{4}$, then by applying \eqref{GNp} with $p=6$ one gets
\begin{equation*}
F(u)\geq \f{1}{2}\|u'\|_2^{2}\left(1-\f{16}{3\pi^{2}}\mu^{2}\right)+\alpha|u(0)|^{2} >0 \quad \forall\,u\in H^{1}_{\mu}(\R^{+}).
\end{equation*}
Furthermore, $F(\la u)\to 0$ as $\la\to 0$, hence $\F(\mu)=0$ for $\mu\leq\f{\sqrt{3}\pi}{4}$. Since positive bound states exist for $\f{\sqrt{3}\pi}{4}<\mu<\f{\sqrt{3}\pi}{2}$ and in this range of masses $\F(\mu)=-\infty$, then ground states do not exist for any value of $\mu>0$.
\end{proof}

\subsection*{Funding acknowledgments}
The work was partially supported by  the INdAM Gnampa 2022 project "Modelli matematici con singolarità per fenomeni di interazione".

\subsection*{Authors contribution}
All authors contributed equally to the manuscript.

\subsection*{Conflict of interest}
All authors declare that they have no conflicts of interest.


\begin{thebibliography}{99}

\bibitem{ABCT-3d}
Adami R., Boni F., Carlone R., Tentarelli L.,
Existence, structure, and robustness of ground states of a NLSE in 3D with a point defect,
\emph{J. Math. Phys.} {\bf 63}, (2022), 071501.

\bibitem{ABCT}
Adami R., Boni F., Carlone R., Tentarelli L.,  
Ground states for the planar NLSE with a point defect as minimizers of the constrained energy, 
\emph{Calc. Var. PDEs} \textbf{61} (5), (2022), art. no. 195.

\bibitem{ABD-20}
Adami R., Boni F., Dovetta S.,
Competing nonlinearities in NLS equations as source of threshold phenomena on star graphs,
\emph{J. Funct. Anal.} {\bf 283}, (2022), 109483.

\bibitem{ABR}
Adami R., Boni F., Ruighi A., 
Non--Kirchhoff vertices and nonlinear Schr\"odinger ground states on graphs, 
{\em Mathematics} {\bf8}(4) (2020), 617.

\bibitem{ACFN-14}
Adami R., Cacciapuoti C., Finco D., Noja D.,
Constrained energy minimization and orbital stability for the NLS equation on a star graph,
\emph{Ann. Inst. H. Poincar\'e Anal. Non Lin\'eaire} {\bf 31}, (2014), no. 6, 1289--1310.

\bibitem{ACFN-16}
Adami R., Cacciapuoti C., Finco D., Noja D.,
Stable standing waves for a NLS on star graphs as local minimizers of the constrained energy,
\emph{J. Diff. Eq.} {\bf 260}, (2016), no. 10, 7397--7415.

\bibitem{ACFN-12}
Adami R., Cacciapuoti C., Finco D., Noja D., 
 Stationary states of NLS on star graphs, 
{\em EPL } {\bf100} (1), (2012), 10003.

\bibitem{ACFN-JDE}
Adami R., Cacciapuoti C., Finco D., Noja D., 
Variational properties and orbital stability of standing waves for NLS equation on a star graph, 
\emph{J. Diff. Eq.} {\bf257} (10), (2014), 3738--3777.

\bibitem{ANV}
Adami R., Noja D., Visciglia N., 
 Constrained energy minimization and ground states for NLS with point defects, 
{\em Discrete Contin. Dyn. Syst. B} \textbf{18} (5), (2013), 1155--1188.

{
\bibitem{AST}
Adami R., Serra E., Tilli P., 
Threshold phenomena and existence results for NLS ground states on graphs, 
\emph{J. Funct. An.} \textbf{271}(1), (2016), 201--223.}


\bibitem{AGHH-88}
Albeverio S., Gesztesy F., Hoegh-Krohn R., Holden H., 
{\em Solvable Models in Quantum Mechanics}, Springer, New York, (1988).

\bibitem{BD2}
Boni F., Dovetta S., 
Doubly nonlinear Schr\"odinger ground states on metric graphs, 
\emph{Nonlinearity} {\bf 35} (7), (2022), 3283.

\bibitem{BD}
Boni F., Dovetta S., 
Ground states for a doubly nonlinear Schr\"odinger equation in dimension one, 
\emph{J. Math. Anal. Appl.} {\bf496} (1), (2021), 124797.

\bibitem{BL}
Brezis H., Lieb E., A relation between pointwise convergence of functions and convergence of functionals, \emph{Proc. Amer. Math. Soc.} \textbf{88}, (1983), no. 3, 486--490.

{
\bibitem{C-18}
Cacciapuoti C.,
Existence of the ground state for the NLs with potential on graphs,
\emph{Contemp. Math.} {\bf717}, (2018), 155--172.

\bibitem{CFN-17}
Cacciapuoti C., Finco D., Noja D.,
Ground state and orbital stability for the NLS equation on a general starlike graph with potentials,
\emph{Nonlinearity} {\bf 30} (8), (2017), 3271--3303.
}

\bibitem{CFN}
Cacciapuoti C., Finco D., Noja D.,
Well posedness of the nonlinear Schr\"odinger equation with isolated singularities,
\emph{J. Diff. Eq.} {\bf 305}, (2021), 288--318.

\bibitem{CM}
Cardanobile S., Mugnolo D., 
Analysis of a FitzHugh--Nagumo--Rall model of a neuronal network, 
{\em Mathematical Methods in the Applied Sciences} {\bf30}(18), (2007), 2281--2308.

\bibitem{CM-95}
Cao Xiang D., Malomed A. B., 
Soliton defect collisions in the nonlinear Schr\"odinger equation, 
\emph{Phys. Lett. A} {\bf206} (1995), 177--182.

\bibitem{CMR} 
Caudrelier V., Mintchev M., Ragoucy E., 
Solving the quantum non-linear Schr\"odinger equation with delta-type impurity, 
{\em J. Math. Phys.} {\bf46} (2005), 042703.


\bibitem{DPS}
Dalfovo F., Giorgini S., Pitaevskii L.P., Stringari S., 
Theory of Bose--Einstein condensation in trapped gases, 
{\em Rev. Mod. Physics} {\bf 71}(3), (1999), 463--512.

{
\bibitem{ET-17}
Erdo\u{g}an M.B., Tzirakis N., 
Regularity properties of the cubic nonlinear Schr\"odinger equation on the half line,
\emph{J. Funct. An.} {\bf 271}(9), (2016), 2539--2568.
}
\bibitem{FN}
Finco D., Noja D.,
Blow-up for the nonlinear Schr\"odinger equation with a point interaction in dimension two,
arXiv:2209.09537 [math.AP] (2022).

{
\bibitem{FHM-17}
Fokas A.S., Himonas A., Mantzavinos D., 
The nonlinear Schr\"odinger equation on the half-line,
\emph{Transactions of the American Mathematical Society}, {\bf369}(1), (2017), 681--709.

\bibitem{FIS-05}
Fokas A.S., Its A. R., Sung L.Y., 
The nonlinear Schr\"odinger equation on the half-line,
\emph{Nonlinearity}, {\bf18} (4), (2005), 1771.
}

\bibitem{FGI-21}
Fukaya N., Georgiev V., Ikeda M., 
On stability and instability of standing waves for 2d-nonlinear Schr\"odinger equations with point interaction,
\emph{J. Diff. Eq.} {\bf 321} (2022), 258--295.

\bibitem{FJ} 
Fukuizumi R., Jeanjean L., 
Stability of standing waves for a nonlinear Schr\"{o}dinger equation with a repulsive Dirac delta potential, {\em Disc. Cont. Dyn. Syst. A} {\bf 21}, (2008), 129--144.

\bibitem{FOO} 
Fukuizumi R., Otha M., Ozawa T., 
Nonlinear Schr\"{o}dinger equation with a point defect, 
{\em Ann. IHP, Analyse non lin\'eaire} {\bf 25}, (2008), 837--845 . 

\bibitem{FT} 
F\"ul\"op T., Tsutsui I., 
A free particle on a circle with point interaction, 
{\em Phys. Lett. A} {\bf 264} (5), (2000), 366--374.

\bibitem{GHW} 
Goodman R.H., Holmes P.J., Weinstein M.I., 
Strong NLS soliton-defect interaction, 
{\em Physica D} {\bf 192}, (2004), 215--248.

\bibitem{GSS-87}
Grillakis M., Shatah J., Strauss W.,
Stability theory of solitary waves in the presence of symmetry. I,
\emph{J. Funct. Anal.} {\bf 74} (1987), no. 1, 160-197. 

\bibitem{HMZ} 
Holmer J., Marzuola J., Zworski M., 
Fast soliton scattering by delta impurities, 
{\em Comm. Math. Phys.} {\bf 274}, (2007), 187--216.

\bibitem{L}
Lannes D., 
{\em The water waves problem : mathematical analysis and asymptotics}, 
Mathematical surveys and monographs, 188, AMS (2013).

\bibitem{L-ANIHPC84}
Lions P.-L.,
The concentration-compactness principle in the Calculus of Variations. The locally compact case. II,
\emph{Ann. Inst. H. Poincar\'e Anal. Non Lin\'eaire} {\bf 36} (1984), no. 4, 223--283.

\bibitem{GMS}
Georgiev V., Michelangeli A., Scandone R.,
Standing waves and global well-posedness for the 2d Hartree equation with a point interaction,
arXiv:2204.05053 [math.AP] (2022).

\bibitem{PP}
Pelinovsky D., Ponomarev D., 
Justification of a nonlinear Schr\"odinger model for laser beams in photopolymers, 
{\em Zeitschrift f\"ur Angewandte Mathematik und Physik} {\bf65}(3) (2014), 405--433.

\bibitem{SBMNU} 
Sobirov Z. A., Babajanov D., Matrasulov D., Nakamura K., Uecker H., 
Sine-Gordon soliton in networks: scattering and transmission at vertices, 
{\em Europhys Lett.} {\bf 115} (2016), 50002. 

\bibitem{SMSN} 
Sobirov Z. A., Matrasulov D., Sawada S., Nakamura K., 
Integrable nonlinear Schr\"odinger equation on simple networks: connection formula at vertices, 
{\em Phys. Rev. E} {\bf 81} (6-2) (2010), 066602.

{
\bibitem{W-05}
Weder R.,
Scattering for the forced non-linear Schrödinger equation with a potential on the half-line,
\emph{Math. Meth. in the Appl. Sc.} {\bf 28}, (2005), 1219--1236.
}



\end{thebibliography}
\end{document}